\documentclass[11pt]{amsart}
\usepackage[margin=1in]{geometry}

\input{epsf.sty}
\usepackage{amsmath,amstext,amsopn,amsfonts,eucal,amssymb}
\input{epsf}
\usepackage{graphicx,wrapfig,url}
\usepackage{color,ulem,cite}
\usepackage[bf,textfont={small}]{caption}

\usepackage{hyperref}

\usepackage{tikz}
\usetikzlibrary{arrows}
\usetikzlibrary{calc}

\let\seq=\subseteq

\def \t {\times}
   \def \rar {\rightarrow}
\def \mc {\mathcal}

\def\Res{\text{RES}}
\def\inn{\text{{\rm in}}}
\def\out{\text{{\rm out}}}
\def\up{\operatorname{Up}}

\def \res{\operatorname{res}}

\def \red {\textcolor{red}}

\def \blue {\textcolor{blue}}

\newtheorem{defn}{Definition}[section]
\newtheorem{prop}[defn]{Proposition}
\newtheorem{lem}[defn]{Lemma}

\newtheorem{thm}[defn]{Theorem}

\newtheorem{example}{Example}[section]
\newenvironment{exam}{\begin{example}\normalfont}{\end{example}}
\newtheorem{remark}[defn]{Remark}
\newenvironment{rem}{\begin{remark}\normalfont}{\end{remark}}
\newtheorem{coro}[defn]{Corollary}

\begin{document}
\baselineskip1.1 \baselineskip

\title{A Graph Isomorphism Condition and Equivalence of Reaction Systems}

\author{Daniela Genova}
\address[D. Genova]{Department of Mathematics and Statistics, University of North Florida, Jacksonville, USA}
\email{d.genova@unf.edu}

\author{Hendrik Jan Hoogeboom}
\address[H.J. Hoogeboom]{LIACS, Leiden University, the Netherlands}
\email{h.j.hoogeboom@liacs.leidenuniv.nl}

\author{Nata\v sa Jonoska}\thanks{This work has been supported in part by the NSF grant CCF-1526485
 and NIH grant R01 GM109459.}
\address[N. Jonoska]{Department of Mathematics and Statistics, University of South Florida, Tampa,  USA}
\email{jonoska@math.usf.edu}

\date{\today}
\begin{abstract}
We consider global dynamics of reaction systems as introduced by Ehrenfeucht and Rozenberg. The dynamics is represented by a directed graph, the so-called transition graph, and two reaction systems are considered equivalent if their corresponding transition graphs are isomorphic. We introduce the notion of a skeleton (a one-out graph) that uniquely defines a directed graph.  We
 provide the necessary and sufficient conditions for two skeletons to define isomorphic graphs. 
 This provides a necessary and sufficient condition for two reactions systems to be equivalent, as well as a characterization of the directed graphs that correspond to the global dynamics of reaction systems. 
\end{abstract}

\maketitle

{\small \noindent {\it keywords}: {directed graphs; graph isomorphism; graphs on posets; dynamics of reaction systems; equivalence of reaction systems}}

\section{Introduction}
Determining whether two graphs are isomorphic is one of the archetypical problems in graph theory and plays an important role in many applications and
network analysis problems. Although there have been
significant advances for this problem in the past year~\cite{graph-iso}, the problem remains difficult.
On the other side, often in network analysis, graphs are partitioned in so called `modules' where
each vertex in a module is adjacent to the same set of vertices outside the module~\cite{def-modules}. Modules in
directed graphs are defined as sets of vertices that have 
 incoming and outgoing edges from, and to, the same vertices
outside the module and it is shown that modular decompositions can be performed in linear
 time~\cite{dir-modules}. In this paper we consider a variation to this notion, i.e., we consider vertices that have the ``same'' incoming edges, and we call such vertices ``companions''. These
 vertices are precisely those that belong to the same region in the Venn  diagram constructed out of
 the family of out-sets (an out-set for $v$ is the set of vertices that have incoming edges starting
  at $v$). We further define a ``skeleton'' of a graph $G=(V,E)$ as a one-out graph over a set $V$ such that  the set of vertices that have non-zero in-degree are representatives of the family of
  out-sets. A skeleton defines uniquely a directed graph and we characterize skeletons of
  isomorphic graphs. Skeletons of isomorphic graphs are called ``companion skeletons''. In particular, skeleton edges swapped at companion vertices produce companion skeletons. This
  observation allows characterizations of reaction systems (described below)
  that exhibit the same global dynamical behavior.

A formal description of biochemical interactions within a confined region bounded with a porous membrane that can interact with the environment has been introduced in \cite{RS}, see \cite{tour-rs} for an overview of the theory.
This formal model, called ``reaction systems'', is based on the idea that each reaction depends on presence of a compound of enzymes, or facilitators, and absence of any other control substance that inhibits the process.
It is assumed further that the reaction is enabled only if the region contains all of the enabling ingredients and none of the inhibitors. In addition, if some ingredients are present in the system, the model allows their presence to be  sufficient to enable all reactions where they participate.
Formally a reaction is modeled as a triple of sets (reactants, inhibitors, results) while the
 reaction system then represents a set of such triples.
In each step, the system produces resulting elements according to the set of reactants that are enabled.
It is further assumed that there is a universal set of elements that can enter the system from the outside environment and interact with the reactants at any given time.
Several studies have addressed the question of the dynamics of the system (the step by step changes of the states of the system), such as
reachability~\cite{rs-reachability}, convergence~\cite{rs-converge},
fixed points and cycles~\cite{rs-fixed,rs-cycle}. It has been observed that the complexity of deciding
existence of certain dynamical properties falls within PSPACE (reachability) or NP-completness (fixed points and fixed point attractors).
In all of these studies, however, the changes in the dynamics through inclusion of new elements entering from the outside environment has not been considered.
We call this condition of no outside involvement within the system as a $0$-context reaction system. In this paper
we study the relationship between the dynamics of the $0$-context reaction systems and the global dynamics of the reaction system that depends on the environmental context. We observe that
quite different  dynamical properties of $0$-context reaction systems produce equivalent global dynamics.

We represent the dynamics of a reaction system as a directed graph where each vertex is a state of the system represented as a set of
elements present at the system at a given time. A directed edge from a vertex terminates at a vertex representing 
 the new state of the system after all reactions enabled at the origin, with possible additions from the outside environment, are performed. 
 In this way, the graph of the
 $0$-context reaction system is a one-out graph (a skeleton) and is a subgraph of the graph of the full dynamics of the system. We characterize the graphs representing the global dynamics of
 reaction systems and show that two reaction systems are equivalent if their $0$-context graphs
 are companion skeletons.

\section{Subsets and Companions}

We denote $[n] = \{0,1,\dots,n-1\}$.
The power set of a set $A$ is denoted by $2^A$.  The number of elements of a finite set $A$ is denoted by $|A|$ and is called the {\it size} of $A$.
Given a function $f: X\rar Y$, the natural equivalence on $X$ defined by $f$ is denoted with $\ker\!\! f$, i.e., $x\ker\!\! f y$ if and only if $f(x)=f(y)$. For $x\in X$ the equivalence class of $\ker\!\! f$ is denoted $[x]_f$.
For a finite set $V$, let $\mc O \subseteq 2^V$ be a family of subsets of $V$. We say that $\mc O$ is a family of sets with domain $V$. The elements in $V$ that appear in the same region of the Venn diagram for $\mc O$ are ``companions" with respect to $\mc O$.
Formally, let
${\mc N_{\mc O}(x)} = \{\; X\in \mc O \mid x\in X \;\}$
be the subfamily containing all sets that include $x$ and ${\mc N^c_{\mc O}(x)}$ its complement in $\mc O$, the subfamily of those sets that don't contain $x$. We call $\mc N_{\mc O}(x)$ the {\em neighborhood} of $x$.

\begin{defn}
Let $V$ be a finite set and $\mc O\subseteq 2^V$ be a family of subsets of $V$. Two elements $x,y\in V$ are \emph{companions with respect to $\mc O$} if ${\mc N_{\mc O}(x)} = {\mc N_{\mc O}(y)}$. We write $x \sim_{\mc O} y$ and denote the equivalence class of $x$ by $C_{\mc O}(x)$.
The set $C_{\mc O}(x)$ is called a {\em companion set}.
\end{defn}

Thus, the equivalence class of every element $x \in V$, the set of companions of $x$ relative to $\mc O$, is the intersection of all sets in $\mc O$ that include $x$ minus the union of the remaining sets in $\mc O$, i.e. $ C_{\mc O}(x) = \bigcap {\mc N_{\mc O}(x)} \setminus \bigcup {\mc N_{\mc O}^c(x)}$. The same
equivalence based on neighborhoods of elements with respect to a family was also used in~\cite{convex-codes} where authors study activity regions for a set of neurons and the convexity of these regions was considered. A special case is when ${\mc N_{\mc O}(x)} = \varnothing$, i.e., when $x \notin \bigcup\mc O$, in which case it is in the outer region, $V \setminus (\bigcup\mc O)$ denoted by $\left(\bigcup\mc O\right)^c$, of the Venn diagram for ${\mc O}$. That is, by convention,  $\bigcap {\mc N_{\mc O}(x)} = \bigcap\varnothing = V$ and $ C_{\mc O}(x) = \left(\bigcup\mc O\right)^c$.

The converse also holds. Any non-empty intersection of sets in $\mc P \subseteq \mc O$ minus the union of the remaining sets $\mc P^c =\mc O \setminus \mc P$ forms an equivalence class $C_{\mc P}$. More precisely,  any non-empty $C_{\mc P}=\bigcap {\mc P} \setminus \bigcup {\mc P^c}$ for some $\mc P \subseteq \mc O$ coincides with $C_{\mc O}(x)$ for some $x \in V$. Assuming $x \in C_{\mc P}$ implies $x \in X$ for every $X \in \mc P$ and $x \not \in Y$ for every $Y \in \mc P^c$. Hence, $x \in \bigcap \mc N_{\mc O}(x) \setminus \bigcup \mc N_{\mc O}^c(x)=C_{\mc O}(x)$ and $C_{\mc P} \subseteq C_{\mc O}(x)$. Conversely, if $y \in C_{\mc O}(x)$ then $y \in \bigcap \mc N_{\mc O}(x)$ which is precisely $\bigcap \mc P$ and $y \not \in \bigcup N_{\mc O}^c(x)=\bigcup \mc P^c$ for $\mc P=\mc N_{\mc O}(x)$ and hence $C_{\mc O}(x) \subseteq C_{\mc P}$. Thus, $C_{\mc P}=C_{\mc O}(x)$.

Therefore, every equivalence class $ C$ of $\sim_{\mc O}$ is characterized by a subset $\mc P \subseteq \mc O$, its neighborhood, such that  $ C = \bigcap {\mc P} \setminus \bigcup {\mc P^c}$. In general, not every $\mc P \subseteq \mc O$ defines an equivalence class, i.e., $\bigcap {\mc P} \setminus \bigcup {\mc P^c}$ might be empty. This is the case when the corresponding region of the Venn diagram of ${\mc O}$ is empty.

For a family of sets $\mc O$ we denote with  ${\mc O}^\cap$ the smallest family of sets that contains $\mc O$ and  is closed under intersection. We say that $\mc O^\cap$ is the {\it intersection} closure of $\mc O$. If $\mc O=\mc O^\cap$ we say that $\mc O$ is {\it intersection closed}.

\begin{exam} \label{ex:companions}
Consider the finite set $V = \{1,2,\dots, 8\}$ and the family of subsets $\mc O \subseteq 2^V$ given by $\mc O = \{\{1,2,3,4\}, \{4,5\},\{5\} \}$. Then $\mc O^\cap = \mc O \cup \{\{ 4 \}\}$. Note that $ C_{\mc O}(1) = \bigcap {\mc N_{\mc O}(1)} \setminus \bigcup {\mc N_{\mc O}^c(1)}=\{1,2,3,4\} \setminus (\{4,5\} \cup \{5\})=\{1,2,3\}=C_{\mc O}(2)=C_{\mc O}(3)$ and $ C_{\mc O}(4) = 
  \{4\}$.  
Thus, the family $\mc O$ defines the following companion sets:
$\{1,2,3\}$,
$\{4\}$,
$\{5\}$, and
$\{6,7,8\}$, 
 where each nonempty region in the corresponding Venn diagram is a companion set.
\end{exam}
%

In sections that follow we use correspondence of families of sets, that have the same sizes of the sets as well as their intersections.

\begin{defn} \label{def:faithful}
Let $\mc O_1$ and $\mc O_2$ be two families of sets over the domains $V$ and $W$ respectively, i.e. $\mc O_1 \subseteq 2^V$ and $\mc O_2 \subseteq 2^W$. A {\em faithful correspondence between $\mc O_1$ and $\mc O_2$} is a
bijection $\eta: {\mc O_1^\cap} \rar {\mc O_2^\cap}$ that satisfies
\begin{itemize}
\item[i] $|X|=|\eta(X)|$ for all $X\in \mc O_1^\cap$,
 and
\item[ii] $\eta (X\cap Y)=\eta (X) \cap \eta(Y)\in \mc O_2^\cap$ for all $X,Y\in \mc O_1^\cap$.
\end{itemize}
\end{defn}

Note that if the domains differ in size, i.e. $|V| \ne |W|$, neither $V$ nor $W$ can be included in $\mc O_1$ (resp. $\mc O_2$).  As shown below in Lemma~\ref{lem:non-empty-companions}, faithful correspondences preserve not only the size of the sets and their intersections, but also the sizes of the companion sets defined by these families. This is only true if the sizes of the domains are equal since otherwise the outer regions differ in size:  $|(\bigcup \mc O_1)^c| \ne |(\bigcup \mc O_2)^c|$.

\begin{lem} \label{lem:non-empty-companions}
Let $\eta$ be a  faithful correspondence between $\mc O_1\subseteq 2^V$ and $\mc O_2\subseteq 2^W$, with $|V|=|W|$.
Then, there exists a bijection $\varphi$ from $V$ to $W$ whose extension to $2^V$ coincides with $\eta$.
Moreover, the extension bijectively maps companion sets of $\mc O_1$  into companion sets of $\mc O_2$, respecting size.
\end{lem}

\begin{proof}
Observe that the equivalence classes of $\sim _{\mc O_1}$, which are companion sets, are precisely the {\it non-empty} sets  $C_{\mc P} = \bigcap {\mc P} \setminus \bigcup {\mc P}^c$ for some family $\mc P \subseteq \mc O_1$, where $\mc P^c=\mc O_1 \setminus \mc P$. Similarly, this holds for $\mc O_2$. More precisely, set $\mc C_i=\{C_{\mc P} \mid \text{ for some } \mc P \subseteq \mc O_i, C_{\mc P} \ne \varnothing\}$ for $i=1,2$. We show that the corresponding equivalence classes (companion sets) under $\eta$ are of the same size. Once this is established, the required bijection that respects the equivalences can naturally be constructed.

We consider the images under $\eta$ of the sets of $\mc O_1$ in $\mc O_2$, and let
$\eta(C_{\mc P}) = C_{\eta(\mc P)} =  \bigcap \eta(\mc P) \setminus \bigcup \eta(\mc P)^c$.
where $\eta(\mc P) = \{\; \eta(X) \mid X\in {\mc P} \;\} \subseteq \mc O_2$ and
$\eta(\mc P^c) = \{\; \eta(X) \mid X\in \mc P^c \;\} = \eta(\mc P)^c$, the latter equality holds because $\eta$ is a bijection.
When $\mc P$ ranges over the subsets of $\mc O_1$, then $C_{\mc P}$ ranges over the equivalence classes of $\mc O_1$ while $\eta(C_{\mc P})$ ranges over the equivalence classes of $\mc O_2$.

We argue that $|C_{\mc P}| = |\eta(C_{\mc P})|$ -- this then takes care of the cases where $\mc P$ defines the empty set instead of an equivalence class, as the corresponding image under $\eta$ is  
void too.
First observe that  $|\bigcap {\mc P} \setminus \bigcup {\mc P^c}| =  |\bigcap {\mc P}| - |\bigcap {\mc P} \cap \bigcup {\mc P^c}|$.

The equality $|\bigcap {\mc P}|=|\bigcap \eta(\mc P)|$ follows from the first and the second requirement in Definition~\ref{def:faithful}, except for the special case when ${\mc P} = \varnothing$, but then $\bigcap {\mc P} = V$ and
consequently, $\eta(\mc P)=\varnothing$ and $\bigcap \eta(\mc P)=V$.
The inclusion-exclusion principle states that we can express the size of a union of sets as sums of sizes of intersections,
$
| \bigcup_{i=1}^m Y_i| =
\sum_{\varnothing \subset I\subseteq [m]} (-1)^{|I|+1}
             | \bigcap_{i\in I} Y_i |
$.
We can apply this to the sets $\bigcap {\mc P} \cap Z$, $Z\in \mc P^c$, to obtain
$|\bigcap {\mc P} \cap \bigcup {\mc P^c}| =
|\bigcap \eta(\mc P) \cap \bigcup \eta(\mc P)^c|$.
Now that corresponding companion sets in $\mc O_1$ and $\mc O_2$ have the same number of elements, we can construct a bijection from $V$ to $W$ that respects companions.
\end{proof}

\section{Directed Graphs and Companion Skeletons}

A {\it directed graph} $G$ is a pair of sets $(V,E)$ where $V$ is a finite set whose
elements are called  {\it vertices} and
$E\subseteq V\t V$ is the set of {\it edges}. We also write $V(G)$ (resp. $E(G)$) to denote the set of vertices
(resp. edges) of $G$.  For an edge $e=(v,v')$ we say that $v$ is the {\it initial} vertex of $e$
and $v'$ is the {\it terminal} vertex of $e$. For a vertex $v\in V$ we define $\inn_G(v)$ to be the set of all vertices that are initial for edges whose terminal vertex is $v$, i.e., $\inn_G (v)=\{\, w \mid (w,v)\in  E\,\}$. Similarly, the {\it out-set} $\out_G(v)$ is the set of
all vertices that are terminal to all edges whose initial vertex is $v$, that is,
$\out_G(v) = \{\, w \mid (v,w)\in E\,\}$.
If $|\out_G(v)|=1$ for all vertices $v$, then we say that $G$ is a \emph{$1$-out} graph.
The {\it out-family} is defined:
$$
\mc{O}(G)=\{\, \out_G(v)\,|\, v \in V\, \}
$$
We drop the subscript $G$ in $\inn_G(v)$ and $\out_G(v)$ whenever the graph is understood from the context.

The following observation is the main motivation for considering the notion of companions.

\begin{lem} \label{in-comp}
Let  $G$ be a directed graph  and let $\mc O=\mc O(G)$. Then $\inn_{G}(x)=\inn_{G}(y)$
if and only if $x\sim_{\mc O} y$.
\end{lem}

\begin{proof}
Observe that  $(w,x)$ is an  edge  in $G$ if and only if  $x\in \out_G(w)$.
Thus two nodes $x$ and $y$ have the same incoming edges if and only if 
they are in the same out-sets in $\mc O$, thus if and only if
they are companions with respect to  $\mc O$, i.e., $x\sim_{\mc O} y$.
\end{proof}

In the case when $\mc O$ is $\mc O(G)$ for a graph $G$, the number of elements in $\mc O$ cannot exceed the number of vertices $V=V(G)$. Therefore, in this case we can always assume that $\mc O$ is indexed by a subset $R_{\mc O}\subseteq V$, i.e., $\mc O = \{ O_z \mid z\in R_{\mc O} \}$. Given this representation, for $z\in R_{\mc O}$ the element $z$ is called the \emph{representative} of $O_z \in \mc O$ and $R_{\mc O}$ is the
\emph{set of representatives} for $\mc O$.
We assume that representatives are unique, i.e., $O_z=O_{z'}$ if and only if $z=z'$.

Formally, representatives are fixed as a bijection $\rho: R_{\mc O}\to \mc O$.
Clearly, the choice of a representative for a set $O$ in $\mc O$ depends on the function $\rho$, but changing the map $\rho$ is the same as renaming the sets in $\mc O$.

\medskip\noindent
{\bf Convention.} In order to ease our notation, if $\mc O$ is understood, we drop the subscript in $R_{\mc O}$ and simply denote the set of representatives for $\mc O$ by $R$. Also, in this context, we always assume that $R$ is a subset of $V$ and
$R$ is an index set for the family $\mc O$ such that
 each $z \in R$ uniquely determines a set $O_z$ in $\mc O$.

\smallskip
Let $G=(V,E)$ be a directed graph where $\out_G(x)\neq \varnothing$ for every vertex $x\in V$, i.e., $G$ is without isolated vertices.
  Let $\mc O =\mc O(G) = \{ O_z \mid z\in R \}$ for a set of representatives $R\seq V$.
Define  $f:V\rar R$  such that $f(x)=z$ if and only if $\out_G(x)=O_z$. Consider a one-out graph  $G_f=(V,E_f)$ where $E_f=\{\;(x,f(x)) \mid x\in V\;\}$.
Then by the choice of $f(x)$ as the representative of $\out_G(x)$ we have that $(x,z),(y,z)\in E_f$ if and only if $\out_G(x)=\out _G(y)$. This induces the following definition.

In the rest of this section we assume that $V$ is fixed and finite.

\begin{defn}
A triple $\sigma=(\mc O,R,f)$ is called a {\em skeleton over set  $V$} if
 $\mc O = \{O_z\mid z\in R\}$ is a family of subsets of $V$ indexed by $R$ with $|\mc O|\le |V|$, and
$f:V\rar R$ is a surjection.  

A {\em graph defined by the skeleton $\sigma$} is the graph $G_\sigma=(V,E)$ where $E=\{\,(x,w) \mid x\in V, w\in O_{f(x)}\}$.
\end{defn}

The graph $G_\sigma$ is uniquely determined by the skeleton $\sigma$. Directly from the definition we have that
$\mc O(G_\sigma)=\mc O$ and $\out_{G_\sigma}(x)=O_{f(x)}$. Moreover, for every directed graph $G$ there is a skeleton $\sigma$ such that
$G=G_\sigma$. It is sufficient to take $\sigma=(\mc O(G), R, f)$ where $R$ is a set of representatives of $\mc O(G)$ and $f(x)=r\in R$ if
and only if $\out_G(x)= O_r$.

\begin{rem}
Note that when every representative $z \in R$ is such that $z \in O_z$, i.e., the representative of every set in $\mc O$ inside that set, then the
one-out graph $G_f=(V,E_f)$ where $E_f=\{\, (x,f(x))\mid x\in V\}$ defined by a skeleton $\sigma=(V,\mc O,f)$ is isomorphic to a subgraph of $G_\sigma$.
Being an element of the set seems as  a natural requirement for the representatives, but unfortunately this is not always possible.
For instance a graph of four vertices cannot have four out-sets from a three element domain, like $\{1,2\}, \{2,3\}, \{1,3\}$ and $\{1\}$.  The out-sets cannot 
contain their own representatives,  and a skeleton for $G$ cannot be a subgraph of $G$.
\end{rem}

\begin{figure}[h]
$G_f$\hskip -1cm
\begin{tikzpicture}
\tikzstyle{punt}=[circle,draw,fill,thick,minimum size=0.7mm,inner sep=0pt]
   \node (1) at (0,-1.4) [punt,label=below:{$1$}] {};
   \node (3) [blue] at (-1.33,-0.43) [punt,label=left:{\blue{$3$}}] {};
   \node (2) [red] at (-0.82,1.13) [punt,label=above:{\red{$2$}}] {};
   \node (4) at (0.82,1.13) [punt,label=above:{$4$}] {};
   \node (5) at (1.33,-0.43) [punt,label=right:{$5$}] {};
\tikzstyle{every edge}=[draw,->,>=stealth,shorten >=5pt,shorten <=5pt]
  \path
       (1) edge (2)
       (3) edge [bend left=10]  (2)
       (2) edge[bend left=10] (3)
       (4) edge  (3)
       (5) edge (3);
       \end{tikzpicture}
          \hskip 1.5cm $G_g$
       \begin{tikzpicture}
\tikzstyle{punt}=[circle,draw,fill,thick,minimum size=0.7mm,inner sep=0pt]
   \node (1) [red] at (0,-1.4) [punt,label=below:{\red{$1$}}] {};
   \node (3) at (-1.33,-0.43) [punt,label=left:{$3$}] {};
   \node (2) at (-0.82,1.13) [punt,label=above:{$2$}] {};
   \node (4) [blue] at (0.82,1.13) [punt,label=above:{\blue{$4$}}] {};
   \node (5) at (1.33,-0.43) [punt,label=right:{$5$}] {};
       \tikzstyle{every edge}=[draw,->,>=stealth,shorten >=5pt,shorten <=5pt]
  \path
       (1) edge [loop,in=-5,out=-45, min distance=1.5cm] (1)
       (3) edge (1)
       (2) edge (4)
       (4) edge [loop,in=-20,out=15, min distance=1.5cm] (4)
       (5) edge (4);
       \end{tikzpicture}
\hskip 1.5cm $G_h$
       \begin{tikzpicture}
\tikzstyle{punt}=[circle,draw,fill,thick,minimum size=0.7mm,inner sep=0pt]
   \node (1) at (0,-1.4) [punt,label=below:{$1$}] {};
   \node (3) at (-1.33,-0.43) [punt,label=left:{$3$}] {};
   \node (2) at (-0.82,1.13) [punt,label=above:{$2$}] {};
   \node (4) [red] at (0.82,1.13) [punt,label=above:{\red{$4$}}] {};
   \node (5) [blue] at (1.33,-0.43) [punt,label=right:{\blue{$5$}}] {};
       \tikzstyle{every edge}=[draw,->,>=stealth,shorten >=5pt,shorten <=5pt]
  \path
       (1) edge (4)
       (3) edge (4) 
       (4) edge (5)
       (5) edge [loop,in=-90,out=-50, min distance=1.5cm] (5)
       (2) edge (5)
;
\end{tikzpicture}

 $G$ \hskip -2cm
\begin{tikzpicture}
\tikzstyle{punt}=[circle,draw,fill,thick,minimum size=0.7mm,inner sep=0pt]
   \node (1) at (0,-1.4) [punt,label=below:{$1$}] {};
   \node (3) at (-1.33,-0.43) [punt,label=left:{$3$}] {};
   \node (2) at (-0.82,1.13) [punt,label=above:{$2$}] {};
   \node (4) at (0.82,1.13) [punt,label=above:{$4$}] {};
   \node (5) at (1.33,-0.43) [punt,label=right:{$5$}] {};
\tikzstyle{every edge}=[draw,->,>=stealth,shorten >=3pt,shorten <=3pt]
  \path
       (1) edge [loop,in=-5,out=-45, min distance=1.5cm] (1)
       (1) edge (2)
       (1) edge [bend left=10] (3)
       (2) edge [loop,in=170,out=135, min distance=1.5cm] (2)
       (2) edge [bend left=10] (3)
       (2) edge [bend left=10] (4)
       (3) edge [bend left=10] (1)
       (3) edge [bend left=10] (2)
       (3) edge [loop,in=-120,out=-80, min distance=1.5cm] (3)
       (4) edge [bend left=10] (2)
       (4) edge [bend left=10] (3)
       (4) edge [loop,in=-20,out=15, min distance=1.5cm] (4)
       (5) edge (2)
       (5) edge (3)
       (5) edge (4)
;
\end{tikzpicture}
\caption{From Example ~\ref{ex:graph-from-skeleton}:
Top: the one-out graphs $G_f$, $G_g$ and $G_h$ corresponding to the skeletons $\nu$, $\sigma$, and $\tau$. In each graph, the representative of $X$ is indicated in \textcolor{red}{red} and the representative of $Y$ is in \textcolor{blue}{blue}.
Bottom: the graph $G=G_\nu=G_\sigma=G_\tau$. Graphs $G_f$ and $G_g$ are subgraphs of $G$, but $G_h$ is not.}
\label{fig:graph-from-skeleton}
\end{figure}
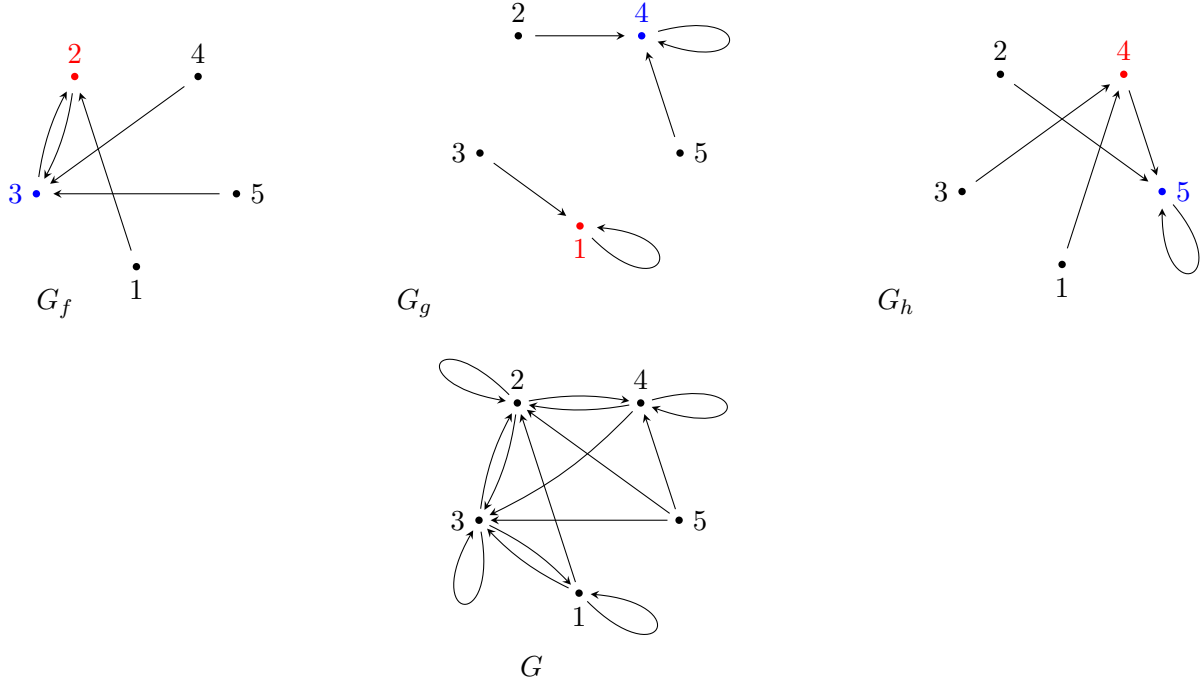

\begin{exam} \label{ex:graph-from-skeleton}
Consider the collection  
$\{ X=\{1,2,3\}, Y=\{2,3,4\} \}$ over $V=\{1,2,\ldots,5\}$. Suppose $G$ is a graph where $\mc O(G)$ is such that $\out(1)=\out(3)=X$
and $\out(2)=\out(4)=\out(5)=Y$.
We consider three skeletons $\nu=(\mc O, R_\nu, f)$, $\sigma=(\mc O',R_\sigma, g)$ and
$\tau=(\mc O'', R_\tau, h)$ where $\mc O=\mc O'=\mc O''$ and the representatives are defined as follows.
$$
\begin{array}{lllll}
R_\nu=\{2,3\} &\text{ with }&X=O_2 &\text{ and }&Y=O_3,\\
R_\sigma=\{1,4\} &\text{ with }&X=O'_1 &\text{ and }&Y=O'_4, \\
R_\tau=\{4,5\} &\text{ with }&X=O''_4 &\text{ and }&Y=O''_5. 
\end{array}
$$
The three one-out graphs $G_f$, $G_g$ and $G_h$ are depicted in Fig.~\ref{fig:graph-from-skeleton}(top).
All three skeletons define the same graph $G$ whose edges are $\{1,3\}\t\{1,2,3\}\cup \{2,4,5\}\t\{2,3,4\}$,
i.e., $G=G_\nu= G_\sigma= G_\tau$. Because both $R_\nu$ and $R_\sigma$ have the property that
the representative of $X$ is in $X$ and the representative of $Y$ is in $Y$, the one-out graphs $G_f$ and $G_g$ are subgraphs of $G$ (see 
Fig.~\ref{fig:graph-from-skeleton}). However, $R_\tau$ doesn't have that property, and $G_h$ is not a subgraph of $G$  
(the vertex $5$ is not in any out-set of $G$).
\qed
\end{exam}

Given two skeletons $\sigma=(\mc O,R,f)$, and $\tau=(\mc P,Q,g)$, we are interested under which conditions their graphs $G_\sigma$ and $G_\tau$ are isomorphic.
We observe that the structures of $G_f$ and $G_g$ may be quite different (as seen in Example~\ref{ex:graph-from-skeleton}) and yet, $G_\sigma$ and $G_\tau$ may be isomorphic.
We utilize the following definition.

\begin{defn} \label{def:graph-companions}
Two skeletons $\sigma=(\mc O,R,f)$ over set $V$, and $\tau=(\mc P,Q,g)$ over set $W$
 are called {\em companions} if there is a bijection $\eta:V\rar W$ that extends to
  a faithful correspondence $\eta:\mc O^{\cap} \rar \mc P^{\cap}$ such that
 $\eta(O_{f(x)})=P_{g(\eta(x))}$.
\end{defn}

Note that because $\eta$ is a bijection and a faithful correspondence, the relationship ``skeleton companions'' is an equivalence relation on skeletons. 
First we see that any pair of skeletons for the same graph are companions.

\begin{lem} \label{lem:same-graph-comp}
Any two skeletons $\sigma=(\mc O,R,f)$ and $\tau=(\mc P,Q,g)$ over $V$ such that $G_\sigma=G_\tau$, are companions.
\end{lem}

\begin{proof}
By definition of $G_\sigma$ we have that $\out_{G_\sigma}(x)=O_{f(x)}$, and $\mc O(G)=
\mc O=\mc P$.
Hence, $\out_G(x)=O_{f(x)}=P_{g(x)}$. Thus the identity map on $V$ extends to
  $\eta:\mc O \rar \mc P$ mapping
$O_{f(x)}$ to $P_{g(x)}$ for all $x\in V$ and it is a faithful correspondence satisfying $\eta(O_{f(x)})=
O_{f(x)}=P_{g(x)}=P_{g(\eta(x))}$.
\end{proof}
Let $\sigma=(\mc O,R,f)$ and $\tau=(\mc P,Q,g)$ be two companion skeletons over $V$, and
$\eta: \mc O^\cap \rar \mc P^\cap$ the corresponding faithful correspondence as in Definition~\ref{def:graph-companions}. Let $C$ be a set of
companions corresponding to $\mc O$, i.e., $x,y\in C$ if and only if $x\sim_{\mc O}y$. Then by Lemma~\ref{lem:non-empty-companions}, $\eta(C)$ is a set of companions corresponding to $\mc P$.
We have the following lemma.

\begin{lem} \label{lem:condition-iii}
For every $x\in V$, $|[x]_f\cap C|=|[\eta(x)]_g\cap \eta(C)|$
\end{lem}
\begin{proof}
We show that  $\eta(X)= Y$ for $X=[x]_f\cap C$ and $Y=[\eta(x)]_g\cap \eta(C)$.
If
 $c\in X$ then $f(c)=f(x)$ and $c\in C$, so $O_{f(c)}= O_{f(x)}$ and therefore
$\eta(O_{f(c)})=O_{g(\eta(c))} =
\eta(O_{f(x)})=O_{g(\eta(x))}$ and so $g(\eta(c)) =g(\eta(x))$ implying $\eta(c)\in [\eta(x)]_g$.
So, $\eta(c)\in Y$. Due to the symmetry of the argument
 (working with $\eta^{-1}$ instead of $\eta$), $\eta(X)=Y$, and because $\eta$
 is a bijection on $V$, $|X|=|Y|$.
\end{proof}

The correspondence  $\eta$ between two companion skeletons in  Definition~\ref{def:graph-companions} can take care of internal symmetry within set $V$. A simple case appears when $\eta$ 
swaps two companion vertices which can be reflected as a swap of the outgoing edges in the corresponding skeletons. Let $\sigma=(\mc O,R,f)$ be a skeleton. Consider two companions $x$ and $y$ with respect to $\mc O$ in $V$ along with the edges $(x,f(x))$ and $(y,f(y))$ in $G_f$. Consider
 the function $f_{x,y}:V\rar V$
  such that $f_{x,y}(x)=f(y)$, $f_{x,y}(y)=f(x)$ and $f_{x,y}(z)=f(z)$ for all $z\not = x,y$,  that is,
$f_{x,y}$ is equal to $f$, except that in $f_{x,y}$ the images of $x$ and $y$ are swapped. Then $\sigma_{x,y}=(\mc O,R,f_{x,y})$, is a skeleton
companion to $\sigma$ through the bijection $\eta:V\rar V$ where $\eta(x)=y$, $\eta(y)=x$ and
$\eta(z)=z$ for $z\not=x,y$.
Note that since $x$ and $y$ are companions,  by Lemma~\ref{in-comp}, these nodes have the same incoming edges in $G_\sigma$, and hence also in $G_{\sigma_{x,y}}$. 
Swapping the out-edges of $x$ and $y$ means that the outgoing edges
 of $x$ and $y$ are interchanged in $G_{\sigma_{x,y}}$ comparing to $G_\sigma$,
  and hence the two graphs are isomorphic (via the isomorphism that swaps vertices $x$ and $y$).
We call $G_{\sigma_{x,y}}$ a result of  $G_\sigma$ through \emph{companion edge swapping}.

\begin{exam} Consider the skeleton $\nu = (\mc O, R_\nu,f)$ from Example~\ref{ex:graph-from-skeleton}; its one-out graph is $G_f$ depicted in Fig.~\ref{fig:graph-from-skeleton}(top-left), and repeated 
in Fig.~\ref{edge-swap}.
The only pair of companions defined by the out-sets $O_2=\{1,2,3\}$ and $O_3=\{2,3,4\}$ is $2,3$. That means the only possibility for companion edge swapping is to exchange the outgoing edges of $2,3$ in $f$. As $2,3$ happen to be mapped to $3,2$ respectively, edge swapping yields the
 one-out graph $G_{f_{2,3}}$ shown in Fig.~\ref{edge-swap}.
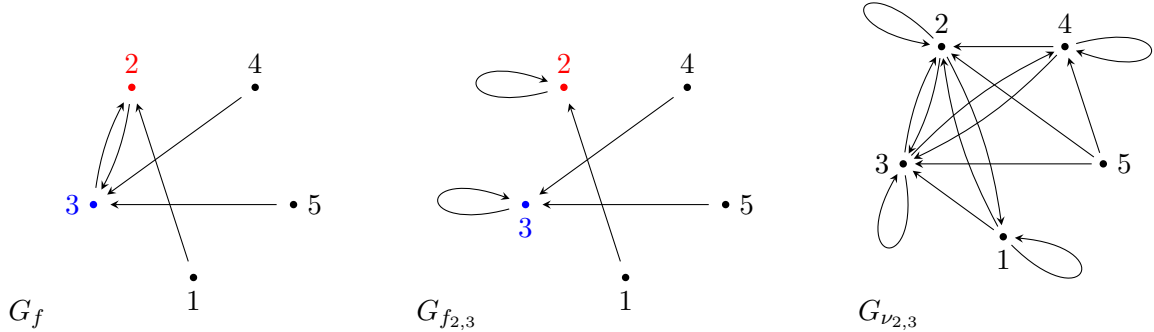
\begin{figure}[h]
\centerline{$G_f$
\begin{tikzpicture}
\tikzstyle{punt}=[circle,draw,fill,thick,minimum size=0.7mm,inner sep=0pt]
   \node (1) at (0,-1.4) [punt,label=below:{$1$}] {};
   \node (3) [blue] at (-1.33,-0.43) [punt,label=left:{\blue{$3$}}] {};
   \node (2) [red] at (-0.82,1.13) [punt,label=above:{\red{$2$}}] {};
   \node (4) at (0.82,1.13) [punt,label=above:{$4$}] {};
   \node (5) at (1.33,-0.43) [punt,label=right:{$5$}] {};
\tikzstyle{every edge}=[draw,->,>=stealth,shorten >=5pt,shorten <=5pt]
  \path
       (1) edge (2)
       (3) edge [bend left=10]  (2)
       (2) edge[bend left=10] (3)
       (4) edge  (3)
       (5) edge (3);
       \end{tikzpicture}
\hskip 1cm
$G_{f_{2,3}}$
\hskip -1cm{\begin{tikzpicture}
\tikzstyle{punt}=[circle,draw,fill,thick,minimum size=0.7mm,inner sep=0pt]
   \node (1) at (0,-1.4) [punt,label=below:{$1$}] {};
   \node (3) [blue] at (-1.33,-0.43) [punt,label=below:{\blue{$3$}}] {};
   \node (2) [red] at (-0.82,1.13) [punt,label=above:{\red{$2$}}] {};
   \node (4) at (0.82,1.13) [punt,label=above:{$4$}] {};
   \node (5) at (1.33,-0.43) [punt,label=right:{$5$}] {};
\tikzstyle{every edge}=[draw,->,>=stealth,shorten >=5pt,shorten <=5pt]
  \path
       (1) edge (2)
       (3) edge [loop,in=160,out=195, min distance=1.5cm] (3)
       (2) edge [loop,in=160,out=195, min distance=1.5cm] (2)
       (4) edge  (3)
       (5) edge (3);
       \end{tikzpicture}
       }
  \hskip 1cm      $G_{\nu_{2,3}}$ \hskip -1.5cm{
\begin{tikzpicture}
\tikzstyle{punt}=[circle,draw,fill,thick,minimum size=0.7mm,inner sep=0pt]
   \node (1) at (0,-1.4) [punt,label=below:{$1$}] {};
   \node (3) at (-1.33,-0.43) [punt,label=left:{$3$}] {};
   \node (2) at (-0.82,1.13) [punt,label=above:{$2$}] {};
   \node (4) at (0.82,1.13) [punt,label=above:{$4$}] {};
   \node (5) at (1.33,-0.43) [punt,label=right:{$5$}] {};
\tikzstyle{every edge}=[draw,->,>=stealth,shorten >=3pt,shorten <=3pt]
  \path
       (1) edge [loop,in=-5,out=-45, min distance=1.5cm] (1)
       (1) edge [bend left =10] (2)
       (1) edge  (3)
       (2) edge [loop,in=170,out=135, min distance=1.5cm] (2)
       (2) edge [bend left=10] (3)
       (2) edge [bend left=10] (1)
       (3) edge [bend left=10] (4)
       (3) edge [bend left=10] (2)
       (3) edge [loop,in=-120,out=-80, min distance=1.5cm] (3)
       (4) edge (2)
       (4) edge [bend left=10] (3)
       (4) edge [loop,in=-20,out=15, min distance=1.5cm] (4)
       (5) edge (2)
       (5) edge (3)
       (5) edge (4)
;
\end{tikzpicture}   }
       }
       \caption{Companion skeletons related by edge swapping. The graph $G_{\nu_{2,3}}$ is isomorphic to $G=G_\nu$ in Fig.~\ref{fig:graph-from-skeleton}.}
       \label{edge-swap}
  \end{figure}     
The skeleton $\nu_{2,3} = (\mc O, R_\nu,f_{2,3})$ differs in the one-out graph, but has 
the same set of representatives for the out-sets. The graphs $G_\nu$ and $G_{\nu_{2,3}}$ 
(see Fig.~\ref{fig:graph-from-skeleton}(bottom) and Fig.~\ref{edge-swap}(right)) are isomorphic.
\qed
\end{exam}

 Fig.~\ref{fig:edge-switching} illustrates   a  case where $\mc O=\mc P$ and $R=Q$ but $f$ and $g$  are distinct.
Then $\eta$ is the identity on $\mc O=\mc P$, but it permutes the elements within a given companion set with respect the family $\mc O$.
By Lemma~\ref{lem:condition-iii}, for every companion set $C$, the number of elements in $C$
that map with $f$ to element $y$ is the same with the number of elements in $C$ that map to
 $y$ with $g$. Fig.~\ref{fig:edge-switching} shows
  an example of $f$ and $g$ on a companion set $C$.
 Because
$O_{y_1}=P_{y_1}$, and $O_{f(x_1)} = \eta(O_{f(x_1)})=P_{g(\eta(x_1))}$, it must be that
$\eta(x_1)=x_2$ or $\eta(x_1)=x_3$. Say it is $x_2$.
Then $O_{y_1}=\eta(O_{y_1})=\eta(O_{f(x_1)})=
O_{g(\eta(x_1)}=O_{g(x_2)}=O_{y_1}$. In this case $G_\sigma$ and $G_\tau$ are isomorphic
and $G_g$ is obtained from $G_f$ by multiple edge swapping.
\medskip

\begin{figure}
\begin{tikzpicture}
\tikzstyle{punt}=[circle,draw,fill,thick,minimum size=0.7mm,inner sep=0pt]
   \draw [rounded corners=10] (0,0.5) -- (1.5,0.5) -- (1.5,4) -- (0,4) --  cycle;
   \node (x6) at (1,1) [punt,label=left:{$x_6$}] {};
   \node (x5) at (1,1.5) [punt,label=left:{$x_5$}] {};
   \node (x4) at (1,2) [punt,label=left:{$x_4$}] {};
   \node (x3) at (1,2.5) [punt,label=left:{$x_3$}] {};
   \node (x2) at (1,3) [punt,label=left:{$x_2$}] {};
   \node (x1) at (1,3.5) [punt,label=left:{$x_1$}] {};
\node at (-0.5,3.5) {$C$};
\node at (2.2,3.75) {$f$};
   \node (y3) at (3,1.2) [punt,label=right:{$y_3$}] {};
   \node (y2) at (3,2.2) [punt,label=right:{$y_2$}] {};
   \node (y1) at (3,3.2) [punt,label=right:{$y_1$}] {};
\tikzstyle{every edge}=[draw,->,>=stealth,shorten >=3pt,shorten <=3pt]
  \path
       (x1) edge (y1)
       (x2) edge (y1)
       (x3) edge (y2)
       (x4) edge (y2)
       (x5) edge (y2)
       (x6) edge (y3)
;
\end{tikzpicture}
\begin{tikzpicture}
\tikzstyle{punt}=[circle,draw,fill,thick,minimum size=0.7mm,inner sep=0pt]
   \draw [rounded corners=10] (0,0.5) -- (1.5,0.5) -- (1.5,4) -- (0,4) --  cycle;
   \node (x6) at (1,1) [punt,label=left:{$x_6$}] {};
   \node (x5) at (1,1.5) [punt,label=left:{$x_5$}] {};
   \node (x4) at (1,2) [punt,label=left:{$x_4$}] {};
   \node (x3) at (1,2.5) [punt,label=left:{$x_3$}] {};
   \node (x2) at (1,3) [punt,label=left:{$x_2$}] {};
   \node (x1) at (1,3.5) [punt,label=left:{$x_1$}] {};
\node at (-0.5,3.5) {$C$}; 
\node at (2.2,3.75) {$g$};
   \node (y3) at (3,1.2) [punt,label=right:{$y_3$}] {};
   \node (y2) at (3,2.2) [punt,label=right:{$y_2$}] {};
   \node (y1) at (3,3.2) [punt,label=right:{$y_1$}] {};
\tikzstyle{every edge}=[draw,->,>=stealth,shorten >=3pt,shorten <=3pt]
  \path
       (x1) edge (y2)
       (x2) edge (y1)
       (x3) edge (y1)
       (x4) edge (y2)
       (x5) edge (y3)
       (x6) edge (y2)
;
\end{tikzpicture}
\caption{Functions $f$ and $g$ of companion skeletons, cf.\ Lemma~\ref{lem:condition-iii}} \label{fig:edge-switching}
\end{figure}
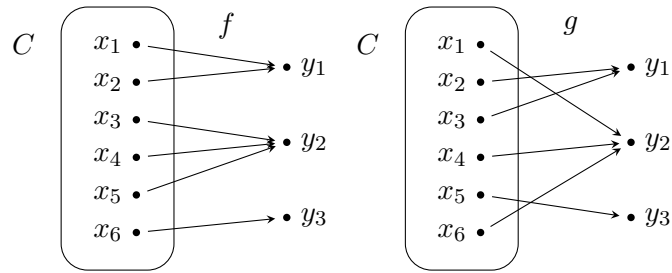

Companion edge swapping is a special case of the following general result.

\begin{thm} \label{thm:faithful-functions}
Two directed graphs $G=(V,E)$ and $G'=(V',E')$ are isomorphic
if and only if there are companion skeletons $\sigma=(\mc O,R, f)$ over set $V$ and
$\tau=(\mc P,Q, g)$ over set $V'$ such that $G=G_\sigma$ and $G'=G_\tau$.
\end{thm}

\begin{proof}
Let $\Psi:G\rar G'$ be an isomorphism. In a natural way we also have a faithful correspondence,
 $\Psi:\mc O(G)^\cap\rar \mc O(G')^\cap$ that is an extension of $\Psi:V\rar V'$. We write $\mc O=\mc O(G)$ and
 $\mc P= \mc O(G')$. Let $R\seq V$ be a set of representatives for $\mc O=\mc O(G)$ and
 define $f:V\rar R$ such that $f(x)=z$ iff $\out_G(x) = O_z$.
  Then $\sigma=(\mc O,R,f)$ is a skeleton and $G=G_\sigma$. Let $Q=\Psi(R)$ and set $g:V'\rar Q$ such that $g(x)=y$ if and only if $f(\Psi^{-1}(x))=\Psi^{-1}(y)$. In other words, for all $x\in V$,
 $\Psi(f(x))=g(\Psi(x))$. We observe that $\tau=(\mc P,Q, g)$ is a skeleton that is companion to
 $\sigma$. By definition of $g$ we have $g(V')=Q$ (if $q\in Q$, then there is $r\in R$ with
 $\Psi(r)=q$, and $x\in V$ with $f(x)=r$, so $g(\Psi(x))=q$) and $Q$ is a set of representatives of
  $\mc P$ by setting $P=P_{q}$ for $P\in \mc P$ if and only if $\Psi(O_r)=P$ and $\Psi(r)=q$.
  Moreover, $\eta=\Psi$ is the faithful correspondence such that $\Psi(O_{f(x)})= P_{\Psi(f(x))}=
  P_{g(\Psi(x))}$.
  Finally we see that $G_\tau=G'$. An edge $(x,y)\in E'$ if and only if $y\in \out_{G'}(x)=P_{q}$
  for some $q\in Q$. Because $\Psi$ is an isomorphism, $\Psi^{-1}(y)\in
  \out_G(\Psi^{-1}(x))=O_{f(\Psi^{-1}(x))}$. So $\Psi^{-1}(q)=f(\Psi^{-1}(x))$, and by definition of $g$,
  $g(x)=q$.

\medskip
 Conversely, suppose $\sigma=(\mc O,R, f)$ and
$\tau=(\mc P,Q, g)$ are companion skeletons and let $G=G_\sigma$ and $G'=G_\tau$.
 Then by Definition~\ref{def:graph-companions} there is a bijection $\eta: V\rar V'$ extending to a faithful correspondence $\eta:
 \mc O^\cap \rar \mc P^\cap$.
We claim that $\eta$ generates an isomorphism $\Psi: G\rar G'$.

Fix $x\in V$.
By Lemma~\ref{lem:condition-iii}  for every companion set $C$ with respect to $\mc O$
 we have that
$| [x]_f\cap C|=| [\eta(x)]_g\cap \eta(C)|$. Let $C=C_{\mc O}(x)$ be the companion set that contains $x$, and so $\eta(C)$ is a companion set with respect $\mc P$ that contains $\eta(x)$. We denote with $X_{x,C}= [x]_f\cap C$ and
$Y_{x,C} = [\eta(x)]_g\cap \eta(C)$.  Therefore, there is a bijection
$\Psi_{x,C}: X_{x,C} \rar Y_{x,C}$.
Observe that the companion classes with respect to $\mc O$ (and similarly $\mc P$) form a
partition of $V$ (also $V'$)
and so the sets $X_{x,C}$  (resp. $Y_{x,C}$) form a partition on $V$ (resp. $V'$).
 We can extend the bijections $\Psi_{x,C}$ to the whole set $V$:
define $\Psi: V\rar V'$ such that
 $\Psi|_{X_{x,C}}=
\Psi_{x,C}$ for all $x\in V$ and companion sets $C$ with respect $\mc O$. Since $\Psi$ is an extension of bijections of a partition of $V$, $\Psi$ is well defined and a bijection itself.
 Observe that by definition of $\Psi$, because $Y_{x,C}\subseteq \eta(C)$, for each $x\in V$, $\Psi(x)$ and $\eta(x)$ belong to the same companion set with respect to $\mc P$.

 It remains to show that if $(x,y)$ is an edge in $G=G_{\sigma}$ then
 $(\Psi(x), \Psi(y))$ is an edge in $G'=G_{\tau}$.
  Let $f(x)=y'$, i.e., $y'\in f(V)$. By the skeleton definition, $y'$ is a representative for  $O=O_{y'}$ where $O\in \mc O$. By definition of $G=G_\sigma$,  $O=O_{f(x)}=\out_G(x)$ and so $y\in O$.
By the definition of $\Psi$,
 $\Psi(x)=\Psi_{x,C}(x)\in Y_{x,C}$, i.e., $\Psi(x)\in [\eta(x)]_g$ and so $g(\Psi(x))=g(\eta(x))$.
 By definition of $G'=G_\tau$, and
 because $G_\tau$ and $G_\sigma$ are companion graphs,
 $\out_{G'}(\Psi(x))=P_{g(\Psi(x))}=P_{g(\eta(x))}=\eta(O_{f(x)})=\eta(O)$. But $\eta(y)\in \eta(O)$ and therefore there is an edge $(\Psi(x),\eta(y))$
 in $G'$.
Then,  by definition of $\Psi$, as observed above,  $\Psi(y)$
belongs to the same companion set as $\eta(y)$. Now by Lemma~\ref{in-comp},
 $\inn_{G}(\eta(y))=\inn_{G'} (\Psi(y))$, hence
there is an edge $(\Psi(x),\Psi(y))$ in $G'=G_\tau$.
With this we conclude that $\Psi$ is an isomorphism from $G$ to $G'$.
 \end{proof}

Isomorphic graphs have companion skeletons, regardless which skeleton we choose. 

\begin{coro} \label{coro:all-skeletons}
Let $G=(V,E)$ and $G'=(V',E')$ be isomorphic graphs. If  $\sigma=(\mc O,R, f)$ and
$\tau=(\mc P,Q, g)$  are skeletons  over set $V$ and $V'$, respectively, such that $G=G_\sigma$ and $G'=G_\tau$ then $\sigma$ and $\tau$ are companions.
\end{coro}

\begin{proof}
As $G$ and $G'$ are isomorphic by Theorem~\ref{thm:faithful-functions} there exist companion skeletons for $G$ and $G'$. By Lemma~\ref{lem:same-graph-comp} these are companions to $\sigma$ and $\tau$ respectively. The result follows by transitivity of companionship of skeletons.
\end{proof}

\section{Partially ordered sets}\label{posets}

In this section we consider partially ordered sets that prepares our discussion on reaction systems
 in the next section. The skeletons and their corresponding graphs have  a  special structure. The nodes are elements from a partially ordered set. Additionally we assume that out-sets of the graphs
  are cones in that partial order. This has a convenient benefit that  the minimal element of a cone 
  can be taken as  a natural representative of that cone.
\medskip

Let $(P,\le)$ be a partially ordered finite set (poset).
For $x\in P$ we define $\up (x)=\{\, y\in P \mid   x \le y \}$ to be the {\it upper cone of $x$} or simply just the {\it cone of $x$}.
 Let $R\subset P$. We define a set of cones {\it based at $R$}
to be  $\up[R]=\{\;\up(x)\,|\, x\in R\;\}$. The poset $P$ is called an {\it upper semi-lattice}  if every subset of $P$ has a least upper bound. Observe that in the case when $P$ is an upper semi-lattice, an intersection of two cones of elements $x$ and $y$ in $P$
 is a cone of the least upper bound $z$ of $\{x,y\}$, i.e.,  smallest $z$  such that $x\le z$ and $y\le z$.
In other words, an upper semi-lattice is closed under intersection of cones.

\medskip\noindent
{\bf Convention.} All posets $P$ considered here are upper semi-lattices.

\medskip
Consider $\mc O_R=\up [R]$ for some $R\subseteq P$, a family of sets consisting of cones based at $R$.
The set of companions of $x\in R$  with respect to $\mc O_R$ becomes $C_{\mc O_R}(x) = C_R(x)= \up(x)\setminus \left( \bigcup_{\substack{y\in  R\\ x\neq y}} \up(y)\right)$.  If $x \in R$, $x$ itself is the smallest element of its  companion equivalence class, and it is called the {\em main representative} of $C_{\mc O_R}(x)$. Observe that for each $x\in R$, $C_{\mc O_R}(x) \cap R=\{x\}$, i.e., no two elements of $R$ are companions, and $R$ is a set of (main) representatives for $\up[R]$.

In particular, when $P$ is the poset $(2^S,\subseteq)$ for some finite set $S$ and  $X\in R\subseteq 2^S$, the set of companions of $X$ with respect to $\mc O_{R}=\up[R]$
 equals
$C_{\mc O_{R}}(X) = \up(X)\setminus \left( \bigcup_{\substack{Y\in R\\ X\subsetneq Y}} \up(Y)\right)$.

\begin{exam} \label{exam:companion}
We consider the poset $(2^S,\subseteq )$ for $S=\{a,b,c\}$.
Let $R=\{\, \{a\}, \{b,c\}\,\}$, and
let $\mc O_{R} = \up[R]$.
Note $\mc O_{R}^\cap = \mc O_{R} \cup \{  \{ a,b,c \}  \}$.
Let $X=\{a\}$.
Then $C_{R}(X)=\{\, \{a\}, \{a,b\}, \{a,c\} \,\}$ as illustrated in Fig.~\ref{fig:set-vs-cones}(left). 
Although $\{a,b,c\}$ is not in $R$, it is no companion for $X$ because it contains a subset $\{b,c\}$ which is in $R$. Note that $\{b,c\}$ has no companions except itself, and $\varnothing$ is a companion to both $\{b\}$ and $\{c\}$.
\qed
\end{exam}

\begin{figure}[h]
\begin{tikzpicture}
\tikzstyle{punt}=[circle,draw,fill,thick,minimum size=0.9mm,inner sep=0pt]
   \draw(-3,2) -- (0,-1)  -- (3,2) ;
   \node at (0,-1)  [punt,label=below:{\small $\varnothing$}] {};
   \draw  [blue]  (-2.75,2.75) -- (-0.75,0.75)  -- (1.5,3) ;
   \node at (-0.75,0.75)  [punt,label=above:{\small $\{a\}$}] {};
   \draw [red] (-1.5,3) -- (0.75,0.75)  -- (2.75,2.75) ;
   \node at (0.75,0.75)  [punt,label=above:{\small $\{b,c\}$}] {};
   \node at (0,1.50)  [punt,label=above:{\small $\{a,b,c\}$}] {};
   \node at (-0.3,-0.2)  [punt,label=above:{\small $\{b\}$}] {};
   \node at (0.3,0)  [punt,label=above:{\small $\{c\}$}] {};
   \node at (-1.2,1.8)  [punt,label=above:{\small $\{a,b\}$}] {};
   \node at (-2.0,2.6)  [punt,label=above:{\small $\{a,c\}$}] {};
\end{tikzpicture} 
\begin{tikzpicture}
\tikzstyle{punt}=[circle,draw,fill=blue!20,thick,minimum size=0.5mm,inner sep=0pt]
   \draw [rounded corners=10] (-2.5,-1.5) -- (4,-1.5) -- (4,2) -- (-2.5,2) --  cycle;
   \draw [blue] (0,0) ellipse (2 and 1);
   \draw [red] (1.5,0.5) ellipse (2 and 1);
   \node at (-1.0,0.1) [punt,label=below:{$1$}] {};
   \node at (-0.5,-0.2) [punt,label=below:{$2$}] {};
   \node at (-0.0,-0.5) [punt,label=below:{$3$}] {};
   \node at (0.8,0.4) [punt,label=below:{$4$}] {};
   \node at (2.5,0.8) [punt,label=below:{$5$}] {};
   \node at (2.5,-0.7) [punt,label=below:{$6$}] {};
   \node at (3.0,-0.7) [punt,label=below:{$7$}] {};
   \node at (3.5,-0.7) [punt,label=below:{$8$}] {};
\end{tikzpicture}
\caption{Venn diagram of the out-sets of a graph, and collection of cones which is in faithful correspondence, see Example~\ref{ex:embedding}} 
\label{fig:set-vs-cones}
\end{figure}
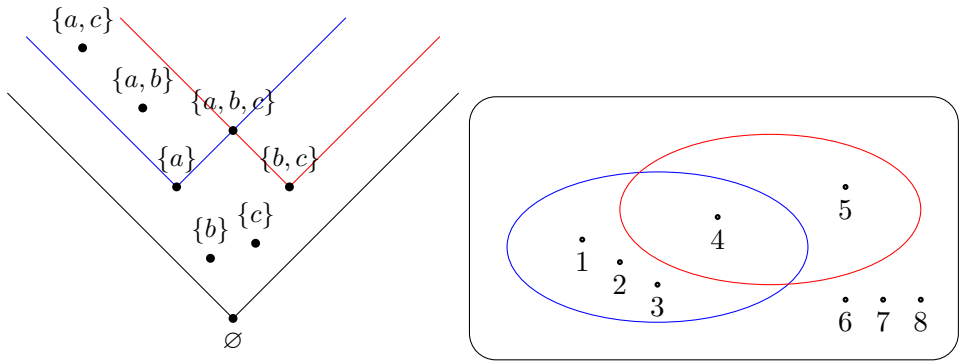

In the remainder of the paper we consider graphs on partial orders.
Let $(P,\le)$ be a poset and $f:P\to P$ a function on $P$. Then $G_f = (P,f)$ is the usual one-out graph associated with $f$. As we have observed, cones in the partial order have a natural representative, so a family of cones has a natural set of representatives.
The \emph{main skeleton} associated with $G_f$ equals $\sigma = (\up[R], R, f)$, with $f(P)=R$. It defines the graph $G_{\sigma}$ with edges $\{ (x,y) \mid x,y\in P, f(x) \le y \}$.
Note that element $x$ is a member of the set $\up(x)$ it represents, and hence $G_f$ is a subgraph of $G_{\sigma}$.

Conversely, if $G$ is a graph with nodes $P$, such that its out-sets are cones in poset $(P,\le)$, then $G$ has a unique main skeleton $\sigma$ such that $G_{\sigma} = G$.

Thus in the case when out-sets are cones,  the main skeleton is fixed by the function $f: P\to P$. This has a consequence for the notion of companions. We say that two functions $f,g: P\to P$ are companions if the main skeletons they define are \emph{companions}, i.e., if there is a bijection $\eta$ on $P$ such that $\eta( \up( f(z)) ) = \up(g(\eta(z)))$ for $z\in P$.
Identifying a function $f$ with its one-out graph $G_f$, we call $G_f$ and $G_g$ {\em companions}
when $f$ and $g$ are.

Summarizing, graphs that have their nodes from a partial order, such that the out-sets all are cones in the partial order have an efficient `summary' where each node is mapped to the minimal element of its out-set. Such a skeleton is unique, given the graph.

\bigskip
We can characterize graphs that are isomorphic to graphs that are defined by (main) skeletons $G_\alpha$. It suffices to consider the structure of the family of out-sets.

\begin{lem} \label{lem:iso-iff-faithful}
Let $P$ be a poset. A graph $G$ is isomorphic to a graph $G_\alpha$,
where $\alpha$ is a main skeleton over $P$ if and only if
$|V(G)| = |P|$ and there exists a faithful correspondence between $\mc O(G)$ and $\up[R]$ for some $R\seq P$.
\end{lem}

\begin{proof}
Obviously the out-sets of the graph $G_\alpha$ consist of upper cones of $P$, hence the forward implication. Conversely, assume $G$ a graph  over $V=V(G)$ with $|P|$ nodes and such that there exists a faithful correspondence $\eta$ between $\mc O(G)$ and $\up[R]$ for some $R\seq P$.
By Lemma~\ref{lem:non-empty-companions} we can find a bijection $\varphi$ between $V(G)$ and $P$ whose extension corresponds to $\eta$.

Let $\sigma = (\mc O(G),Q,f)$ be a skeleton for $G$.
We define a main skeleton $\alpha = (\up[R],R,g)$ over $P$ which is a companion to $\sigma$; from this the result follows by the characterization in Theorem~\ref{thm:faithful-functions}.

For each $x\in V$ the out-set $\out(x)= O_{f(x)}$ of $x$ in $G$ corresponds via $\eta$, and $\varphi$, to a cone $\up(z)$ for some $z\in P$. Now define $g(\varphi(x))$ to be $z$. Then $\sigma$ and $\alpha$ are companions by construction. 
\end{proof}

\begin{exam} \label{ex:embedding}
Let graph $G$ with node set $V = \{1,2,\dots, 8\}$ be given by the following adjacency matrix.
%
$$
\begin{array}{c|*8c}
  & 1 & 2 & 3 & 4 & 5 & 6 & 7 & 8   \\\hline
 1 &   &   &   & \times & \times &   &   &    \\
 2 & \times & \times & \times & \times &   &   &   &    \\
 3 & \times & \times & \times & \times & \times & \times & \times & \times  \\
 4 &   &   &   & \times & \times &   &   &  \\
 5 & \times & \times & \times & \times &   &   &   &    \\
6 & \times & \times & \times & \times & \times & \times & \times & \times \\
 7 &   &   &   & \times & \times &   &   &   \\
 8 & \times & \times & \times & \times &   &   &   &  \\
\end{array}
$$
%
Then $\mc O(G) = \{ V, \{1,2,3,4\}, \{4,5\} \}$,
and $\mc O^\cap(G) = \mc O(G) \cup \{\{ 4 \}\}$ (see Fig.~\ref{fig:set-vs-cones}(right)).
 Now $\mc O(G)$ defines the following companion sets:
 $\{1,2,3\}$, $\{4\}$, $\{5\}$, and $\{6,7,8\}$.

There is a faithful correspondence $\eta$ between $\mc O(G)$ and a family of upper cones $R$ within $P = ( 2^{\{a,b,c\}},\subseteq )$. Take $R = \{\, \varnothing, \{a\},\{b,c\}, \{a,b,c\} \,\}$.  Then $\up[R] =  \{ \up(\varnothing), \up(\{a\}), \up(\{b,c\}, \up(\{a,b,c\}\}$ contains sets of size $8$, $4$, $2$ and $1$, respectively, matching those in $\mc O^\cap(G)$.

Sets in $\mc O(G)$ and $\up[R]$ are illustrated in Fig.~\ref{fig:set-vs-cones}.
Let $\varphi:V\to 2^{\{a,b,c\}}$ be the bijection as shown in the table below,
extended to sets matching $\eta$.

\medskip
\centerline{$
\begin{array}{c|ccc|c|c|ccc}
x  & 1 & 2 & 3 & 4 & 5 & 6 & 7 & 8   \\\hline
\varphi(x) & \{a\} & \{a,b\} & \{a,c\} & \{a,b,c\} & \{b,c\} & \varnothing & \{b\} & \{c\}
\\
g(\varphi(x))=z & \{b,c\} & \{a\} & \varnothing & \{b,c\} & \{a\} & \varnothing & \{b,c\} & \{a\}
\end{array}
$}
\medskip

The bottom  rows of the above table define the main skeleton on $P = ( 2^{\{a,b,c\}},\subseteq )$ that defines a graph isomorphic to $G$.
In this table we can swap the elements $\varphi(x)$ within each companion set, and obtain the 
 main skeleton for a different but an isomorphic graph. 
\qed
\end{exam}

\section{Reaction Systems}

We start by recalling some basic notions of reaction systems \cite{RS}.
A reaction is formalized as a triplet that represent the {\it reactant}, {\it inhibitor} and {\it product}, respectively. Whenever all reactants and none of the inhibitors are present, the reaction will yield the product.  The effect of separate reactions is cumulative, the union of the products for applicable reactions. More precisely, we have the following.

\begin{defn}
A \emph{reaction system} (RS) is a pair $\mathcal A = (S, A)$ where $S$ is a finite set, the \emph{background set}, and $A \subseteq (2^S\setminus \{\varnothing\}) \times (2^S\setminus \{\varnothing\}) \times2^S$ is a set of \emph{reactions} in $S$.

Let $X\subseteq S$.
For a reaction $a=(R,I,P)$ we say that $a$ is \emph{enabled in} $X$ iff $R\subseteq X$ and $I \cap X = \varnothing$. The \emph{result} of $a$ on $X$, denoted $\res_a(X)$, equals $P$ if $a$ is enabled in $X$, and $\varnothing$, otherwise.
The result of $X$ in $\mathcal A$ equals $\res_{\mathcal A}(X) = \bigcup_{a\in A} \res_a(X)$.
\end{defn}

Note that it is required that reactant and inhibitor are non-empty. This technical assumption has the consequence that no reaction is enabled in either $\varnothing$ or $S$, thus $\res_{\mathcal A}(\varnothing) =\res_{\mathcal A}(S) = \varnothing$.

Given a reaction system $\mc A=(S,A)$ define $\Res_{\mc A}\subseteq 2^S$ to be the set of all 
$Y\subseteq S$ such that there is $X$ with $\res_{\mc A}(X)=Y$. Note that 
$\varnothing \in \Res_{\mc A}$. Thus $\res_{\mc A}:2^S\rar \Res_{\mc A}$ is a surjection.

\begin{exam} \label{ex:tour-example}
We use the reaction system $\mc A = (S,A)$ from \cite[Example 7]{tour-rs}.
It has background set $S = \{1,2,3,4\}$, and six reactions belong to $A$:

\noindent
$a_1 = ( \{ 1 \},\{ 3 \},\{ 2 \} )$,
$a_2 = ( \{ 2 \},\{ 1 \},\{ 1 \} )$,
$a_3 = ( \{ 2 \},\{ 3 \},\{ 3 \} )$,
$a_4 = ( \{ 3 \},\{ 1,2 \},\{ 1,2,4 \} )$,
\\
$a_5 = ( \{ 4 \},\{ 3 \},\{ 1,2 \} )$, and
$a_6 = ( \{ 1,3 \},\{ 2,4 \},\{ 2,3 \} )$.

In $\{2,3,4\}$ only $a_2$ is enabled, so we have $\res_{\mc A}(\{2,3,4\}) = \{1\}$.
In $\{1,2,3\}$ no reactions are enabled, so $\res_{\mc A}(\{1,2,3\}) = \varnothing$.
In $\{1,2,4\}$ all three $a_1$, $a_3$ and $a_5$ are enabled, so $\res_{\mc A}(\{1,2,4\}) = \{1,2,3\}$.
\qed
\end{exam}

\begin{figure}[h]
\begin{tikzpicture}[scale =0.7]
\tikzstyle{every edge}=[draw,->,>=stealth]
  \node (0) at (9.0,4.0) {$\varnothing$} ;
  \node (a) at (4.0,3.0) {$\{1\}$} ;
  \node (b) at (6.5,2.0) {$\{2\}$} ;
  \node (c) at (16.0,2.0) {$\{3\}$} ;
  \node (d) at (0.0,3.0) {$\{4\}$} ;
  \node (ab) at (2.0,4.0) {$\{1,2\}$} ;
  \node (ac) at (6.5,4.0) {$\{1,3\}$} ;
  \node (ad) at (-0.5,5.0) {$\{1,4\}$} ;
  \node (bc) at (4.0,5.0) {$\{2,3\}$} ;
  \node (bd) at (14.0,5.0) {$\{2,4\}$} ;
  \node (cd) at (16.5,4.0) {$\{3,4\}$} ;
  \node (abc) at (11.0,4.0) {$\{1,2,3\}$} ;
  \node (abd) at (14.0,3.0) {$\{1,2,4\}$} ;
  \node (acd) at (11.0,2.0) {$\{1,3,4\}$} ;
  \node (bcd) at (2.0,2.0) {$\{2,3,4\}$} ;
  \node (S) at (11.0,6.0) {$\{1,2,3,4\}$};
  \path[every node/.style={draw=none}]
       (0) edge [loop,in=135,out=90, min distance=1.5cm]  (0)
       (a) edge (b)
       (b) edge (ac)
       (c) edge (abd)
       (d) edge (ab)
       (ab) edge (bc)
       (ac) edge (bc)
       (ad) edge (ab)
       (bc) edge (a)
       (bd) edge (abc)
       (cd) edge (abd)
       (abc) edge (0)
       (abd) edge (abc)
       (acd) edge (0)
       (bcd) edge (a)
       (S) edge (0)
           ;
\end{tikzpicture}
\caption{Zero context graph $G_{\mc A}^0$ for the RS from Example~\ref{ex:tour-example}.}\label{fig:zero-context}
\end{figure}
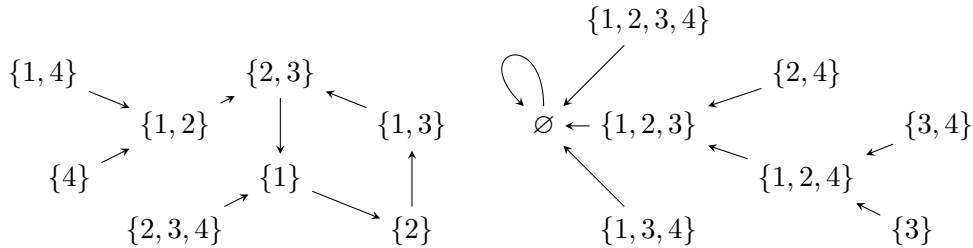

The dynamic behaviour of a reaction system is given by the notion of state sequence of an interactive process. Let $S$ be the background set. Then a {\it state sequence} is of the form $W_0,W_1, \dots, W_n$, where $W_i\subseteq S$, is such that $\res( W_{i}) \subseteq  W_{i+1}$ for $0\le i< n$. The intuition behind this computational process is as follows. In each step the new products are generated by the enabled reactions. Eements that are not produced by the reactions vanish, the so-called principle of {\it non-permanency}. On the other hand, in each step the context, the environment in which the reactions take place, may add new elements in the state of the system. Hence the 
new state of the system is a step from $W_{i}$ to any superset of $\res(W_{i})$.

From this perspective we introduce two graphs to represent the stepwise behavior of
 reaction systems,  without and with context.

\begin{defn}
For a RS $\mc A$ the {\em $0$-context graph of $\mc A$}  is the one-out graph
 $G_{\mc A}^0=(2^S,E)$
with edge set $E= \{\, (v,\res_{\mc A}(v)) \,|\, v\in 2^S\,\}$.

For a RS $\mc A$ the {\em transition graph of  $\mc A$} is the graph
 $G_{\mc A}=(2^S,E)$ with edge set $E= \{\, (v,w) \mid v\in 2^S,  \res_{\mc A}(v) \subseteq w \,\}$.
\end{defn}

When $\mc A$ is understood, we allow to drop the subscript in $\res$.

By definition the 0-context graph $G_{\mc A}^0$ is a subgraph of the transition graph $G_{\mc A}$ of the same reaction system. 

The link to one-out graphs $G_f$ and graphs $G_\alpha$ defined by the main
 skeleton $\alpha$ as defined in Section~\ref{posets}  is obtained as follows.
Defined on node set $2^S$, the $0$-context graph $G^0_{\mc A}$ equals the one-out graph $G_{\res_{\mc A}}$ defined by $\res_{\mc A}$, while the transition graph $G_{\mc A}$ is the graph defined by the main skeleton fixed by $\res_{\mc A}$ in the partial order $(2^S,\subseteq)$.

\begin{exam}
For the RS $\mc A$ from Example~\ref{ex:tour-example} the 0-context graph $G_{\mc A}^0$ is given in Figure~\ref{fig:zero-context}. The family $\Res_{\mc A}$ matches the family $\mc R$ of sets that have nonempty set of incoming edges. 
\qed
\end{exam}

As $G^0_{\mc A}$ is a one-out graph, it consists of one or more components, each of these components is tree-like, `ending in' a single cycle. By definition of reaction systems 
one component must have a loop at  $\varnothing$.
It turns out that virtually any graph on domain $2^S$ is a $0$-context graph of a reaction system. We only have to respect the special position of the minimal and maximal set $\varnothing$ and $S$.
This follows from a common construction in reaction systems, see, e.g., the implementation of a transition system in Section~4.2 of \cite{tour-rs}.

\begin{prop} \label{prop:one-out}
A one-out graph $G$ with vertex set $2^S$ is a 0-context graph of a RS if and only if there are two edges
 $(\varnothing,\varnothing), (S,\varnothing)$ in $G$.
\end{prop}

\begin{proof}
For every RS we have $\res( \varnothing ) = \res ( S ) = \varnothing$ as no reactions are enabled in the empty set, and all reactions are inhibited in the full set $S$ due to requirement that reactant and inhibitor are non-empty. Therefore, every $0$-context graph of a RS contains edges  $(\varnothing,\varnothing), (S,\varnothing)$.

Consider a one-out graph with vertices $2^S$ containing edges  
$(\varnothing,\varnothing), (S,\varnothing)$.
Then we define a set of reactions matching the rest of the edges in the graph:
$A = \{\; (X, S\setminus X, Y) \mid (X,Y)\in E, X\neq \varnothing, S \;\}$. The complementarity of the first and second component of the reactions ensures that each reaction is enabled only at a single set, and that
$(X,Y)\in E$ if and only if $Y=\res(X)$.
\end{proof}

As a consequence of Lemma~\ref{lem:iso-iff-faithful} we can characterize graphs that are isomorphic to transition graphs of reaction systems.
Recall that a faithful correspondence maps companion sets between two families of sets, respecting their sizes.
As observed in Lemma~\ref{lem:iso-iff-faithful}, the out-sets of transition graphs must faithfully correspond to the structure of upper cones in $(2^S,\subseteq)$. 
Additionally, by Proposition~\ref{prop:one-out}, each transition graph must have edges $(\varnothing,\varnothing)$, $(S,\varnothing)$.
Similar edges must be present in any graph isomorphic to a transition graph, one in the intersection of all out-sets, the other in none of the out-sets (except the out-set that consists of all vertices). This characterization is given formally in the following thoerem.

\begin{thm} \label{thm:lattice-embedding}
A  graph $G=(V,E)$ is isomorphic to a transition graph of a reaction system if and only if 
\begin{enumerate}
\item
$|V|=2^n$ for some $n\ge 1$, 
\item
there is a
faithful correspondence $\eta$ between $\mc O(G)$ and a family of upper cones of $(2^{[n]},\subseteq)$,
\item
there is a vertex $v_{\bot}\in  
 V\setminus \bigcup_{\substack{X\in \mc O(G)\\ X\neq V}}$, and a vertex $v_\top\in \bigcap\mc O(G)$ such that $(v_{\bot},v_{\bot}), (v_{\top},v_{\bot})$ are edges in $G$.
\end{enumerate}
\end{thm}

\begin{proof}
If $\mc A$ is a RS with background set $S$, and taking $v_\bot = \varnothing$ and $v_\top = S$,  requirements (1) to (3) by construction hold for the transition graph $G_{\mc A}$, so must hold for any graph isomorphic to it.

\smallskip
Assume $G$ is a graph as given in the statement. We show it is isomorphic to a transition graph of a RS.
Let $S = [n]$.
By (1,2) $G$ is isomorphic to a graph $G_\alpha$ over the poset $(2^S,\seq)$, where $\alpha =(\up[R], R, g)$ is a main skeleton. 
In order for $G_\alpha$ to be a transition graph, the function $g: 2^S \to 2^S$ must additionally satisfy $g(\varnothing) = \varnothing$ and $g(S) = \varnothing$, cf.\ Proposition~\ref{prop:one-out}.

Note that the skeleton $\alpha$ as constructed is based on the node to node bijection $\varphi$ between $V=V(G)$ and $2^S$ that is extending the set to a set bijection $\eta$ between $\mc O(G)$ and  $\up[R]$, see Lemma~\ref{lem:non-empty-companions}.
In constructing this bijection $\varphi$ there is freedom, as long as we respect companion sets.

Note that the incoming edge $(v_\top,v_\bot)$ to $v_\bot$ means that $v_\bot$ is an element of one of the out-sets. As we have chosen $v_\bot$ to be in none of the out-sets except $V(G)$ we know that $v_\bot$ only belongs to the companion set that is within $V(G)$ and none of the other sets from $\mc O(G)$. That companion set must match the same `outer' companion set $C_{\up[R]}(\varnothing)$.  That means we can take that $\varphi(v_\bot) = \varnothing$.
At the same time the only out-set that contains $v_\bot$ must be $V(G)$, so we conclude that in $G$ both $v_\bot$ and $v_\top$ have out-set $V(G)$. 

Similarly the intersection of all out-sets is a companion set which must correspond to the `inner' companion set $\bigcap \up[R]$, which contains $S$, and we may assume that $\varphi(v_\top) = S$. 

To conclude we follow the
 proof of Lemma~\ref{lem:iso-iff-faithful}. For each $x\in V(G)$ with out-set $X$ in that proof we set $g(\varphi(x)) = z$ where $\up(z)=\varphi(X)$ is the set in $\up[R]$ that corresponds to $X$. If we apply this to $v_\bot$ we set $g(\varphi(v_\bot)) = g(\varnothing) = \varnothing$ as $v_\bot$ has 
 out-set $V(G)$ which must correspond to $\up(\varnothing)$. Same holds for $v_\top$, thus $g(S) = \varnothing$, as required.
\end{proof}

As an immediate application of Theorem~\ref{thm:faithful-functions} we can characterize when reaction systems have isomorphic transition graphs.

\begin{thm}
For reaction systems $\mc A$ and $\mc A'$, their
transition graphs $G_{\mc A}$ and $G_{\mc A'}$ are isomorphic
if and only if the $0$-context graphs $G^0_{\mc A}$ and $G^0_{\mc A'}$ are companions.
\end{thm}

We see that there is no obvious structural relationship between two RS's $\mc A$ and $\mc A'$ such that $G_{\mc A}$ and $G_{\mc A'}$ are isomorphic. A basic operation on $0$-context graphs that yields a transition graph isomorphic to that of the original $0$-context graph is
based on the  companion edge swapping on the main skeletons.

Let $\mc A$ be a RS, and
let $x,y\in 2^S$ be two elements such that $x\sim_{\Res_{\mc A}} y$. Consider the pair of edges $(x,\res_{\mc A}(x) )$ and  $(x',\res_{\mc A}(x') )$ in $G^0_{\mc A}$. The graph $G_{x,x'}$ that is obtained from $G^0_{\mc A}$ by swapping the targets of these edges, introducing the new pair  $(x,\res_{\mc A} (x'))$ and  $(x',\res_{\mc A} (x) )$, is again a one-out graph and hence the $0$-context graph $G^0_{\mc A'}$ of a RS ${\mc A'}$ (provided $x,x'$ are unequal to both $\varnothing$ and $S$, see Proposition~\ref{prop:one-out}).

As seen in Section~\ref{posets}, by Lemma~\ref{in-comp} the incoming edges of $u$ and $v$ in $G_{\mc A}$ are equal. Switching outgoing edges of $x$ and $x'$ in $G^0_{\mc A}$ swaps all outgoing edges of $x$ and $x'$ in $G_{\mc A}$. All other vertices in $G^0_{\mc A}$ have
 their edges unchanged, so we can conclude that $G_{\mc A}$ and  $G_{\mc A'}$ are isomorphic.

\begin{exam} \label{ex:edge-swap}
Reconsider the RS ${\mc A}$ from Example~\ref{ex:tour-example}, see Figure~\ref{fig:zero-context} for its $0$-context graph $G^0_{\mc A}$. The elements $\{1,3\}$ and $\{1,3,4\}$ are companions with respect to $\Res_{\mc A}$. After companion edges switching we obtain a $0$-context graph  $G_{\mc A'}$ with a single component; it has no cycles except for the (unavoidable) loop at $\varnothing$. The original $G_{\mc A}$ has two components.
The transition graphs $G_{\mc A}$ and $G_{\mc A'}$ are isomorphic.
\qed
\end{exam}

\begin{figure}[h]
\begin{tikzpicture}[scale =0.7]
\tikzstyle{every edge}=[draw,->,>=stealth]
  \node (0) at (9.0,4.0) {$\varnothing$} ;
  \node (a) at (4.0,3.0) {$\{1\}$} ;
  \node (b) at (6.5,2.0) {$\{2\}$} ;
  \node (c) at (16.0,2.0) {$\{3\}$} ;
  \node (d) at (0.0,3.0) {$\{4\}$} ;
  \node (ab) at (2.0,4.0) {$\{1,2\}$} ;
  \node (ac) at (6.5,4.0) {$\{1,3\}$} ;
  \node (ad) at (-0.5,5.0) {$\{1,4\}$} ;
  \node (bc) at (4.0,5.0) {$\{2,3\}$} ;
  \node (bd) at (14.0,5.0) {$\{2,4\}$} ;
  \node (cd) at (16.5,4.0) {$\{3,4\}$} ;
  \node (abc) at (11.0,4.0) {$\{1,2,3\}$} ;
  \node (abd) at (14.0,3.0) {$\{1,2,4\}$} ;
  \node (acd) at (11.0,2.0) {$\{1,3,4\}$} ;
  \node (bcd) at (2.0,2.0) {$\{2,3,4\}$} ;
  \node (S) at (11.0,6.0) {$\{1,2,3,4\}$};
  \path[every node/.style={draw=none}]
       (0) edge [loop,in=135,out=90, min distance=1.5cm]  (0)
       (a) edge (b)
       (b) edge (ac)
       (c) edge (abd)
       (d) edge (ab)
       (ab) edge (bc)
       (ac) edge [dashed] (0)
       (ad) edge (ab)
       (bc) edge (a)
       (bd) edge (abc)
       (cd) edge (abd)
       (abc) edge (0)
       (abd) edge (abc)
       (acd) edge [out=170, in=-45, dashed] (bc)
       (bcd) edge (a)
       (S) edge (0)
           ;
\end{tikzpicture}
\caption{The $0$-context graph for a RS equivalent to the one in  Example~\ref{ex:tour-example} after swapping companion edges. Dashed edges are the replacements of the two original edges (compare to Fig.~\ref{fig:zero-context}).
} \label{fig:zero-context-swapped}
\end{figure}
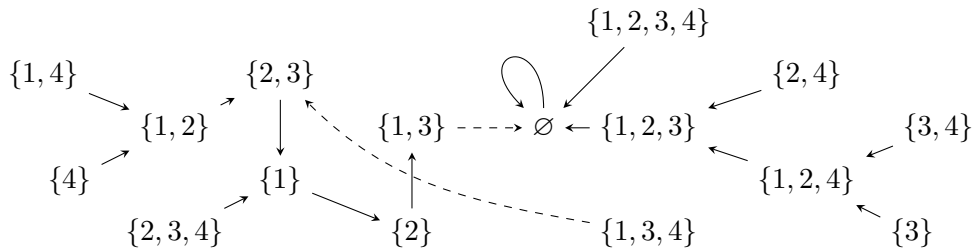

%
%
%
%
%
%
%
%


\end{document}